\documentclass[10pt,a4paper]{amsart}
\usepackage[english,french]{babel}
\usepackage[latin1]{inputenc}
\usepackage{color}
\usepackage{amsfonts}
\usepackage{amssymb}
\usepackage{newlfont}
\usepackage{mathrsfs}
\usepackage{float}
\usepackage{bm}

\definecolor{liens}{rgb}{1,0,0}
\usepackage[colorlinks=true, linkcolor=blue, 
hyperfootnotes=true,citecolor=blue,urlcolor=black]{hyperref}

\newtheorem{theo}{Theorem}[section]

\newtheorem{prop}[theo]{Proposition}

\newtheorem{lem}[theo]{Lemma}
\newtheorem{coro}[theo]{Corollary}

\theoremstyle{definition}
\newtheorem{defi}[theo]{Definition}

\theoremstyle{remark}
\newtheorem{example}[theo]{Example}
\newtheorem{remark}[theo]{Remark}

\numberwithin{equation}{section}

\newcommand{\N}{\mathbb N}
\newcommand{\Z}{\mathbb Z}

\newcommand{\C}{\mathbb C}
\def\K{\mathbf{K}}
\def\L{\mathbf{L}}

\def\de{\delta}
\def\s{\sigma}

\def\c{\mathcal{C}}
\def\d{\delta}

\newcommand{\Gal}{\operatorname{Gal}}
\newcommand{\GL}{\operatorname{GL}}
\newcommand{\SL}{\operatorname{SL}}

\newcommand{\divi}[1]{\operatorname{div}_{#1}}
\newcommand{\weight}[1]{\omega_{#1}}
\newcommand{\degr}[1]{\deg_{#1}}

\newcommand{\mero}{\operatorname{M}}

\newcommand{\prodtheta}[1]{\Theta_{#1}}
\newcommand{\quotprodtheta}[1]{\Theta_{#1}^{quot}}

\newcommand{\propP}{(\mathcal{P})}
\newcommand{\subeps}{\bm{\varepsilon}}

\usepackage[all]{xy}

\title[Diff. transc. criteria \& elliptic hypergeometric functions]{Differential transcendence criteria for second-order linear difference equations and elliptic hypergeometric functions}

\author{Carlos E. Arreche}
\thanks{The first author was partially supported by the National Science Foundation under Grant No. CCF-1815108. This work was supported in part by the ANR \href{https://specfun.inria.fr/chyzak/DeRerumNatura/}{\emph{De rerum natura}} project, grant ANR-19-CE40-0018 of the French \emph{Agence Nationale de la Recherche}.}
\address{The University of Texas at Dallas, Mathematical Sciences FO 35, 800 West Campbell Road, Richardson, TX 75024, USA}
\email{arreche@utdallas.edu}
\author{Thomas Dreyfus}
\address{Institut de Recherche Math\'ematique Avanc\'ee, U.M.R. 7501 Universit\'e de Strasbourg et C.N.R.S. 7, rue Ren\'e Descartes 67084 Strasbourg, FRANCE}
\email{dreyfus@math.unistra.fr}
\author{Julien Roques}
\address{Univ Lyon, Universit\'e Claude Bernard Lyon 1, CNRS UMR 5208, Institut Camille Jordan, 43 blvd. du 11 novembre 1918, F-69622 Villeurbanne cedex, France}
\email{roques@math.univ-lyon1.fr}

\begin{document}
\selectlanguage{english}
\sloppy

\begin{abstract}
We develop general criteria that ensure that any non-zero solution of a given second-order difference equation is differentially transcendental, which apply uniformly in particular cases of interest, such as shift difference equations, $q$-dilation difference equations, Mahler difference equations, and elliptic difference equations. These criteria are obtained as an application of differential Galois theory for difference equations. We apply our criteria to prove a new result to the effect that most elliptic hypergeometric functions are differentially transcendental.
\end{abstract}

\date{\today}

\maketitle
\tableofcontents

\pagebreak
\section{Introduction}

The differential Galois theory for difference equations developed in \cite{HS} provides a theoretical tool to understand the differential-algebraic properties of solutions of linear difference equations. Given a $\sigma\delta$-field $K$ (i.e., $K$ is equipped with an automorphism $\sigma$ and a derivation $\delta$ such that $\sigma\circ\delta=\delta\circ\sigma$), one considers an $n^\text{th}$ order linear difference equation of the form
\begin{equation}\label{intro:eq} \sigma^n(y)+a_{n-1}\sigma^{n-1}(y)+\dots+a_1\sigma(y)+a_0y=0,\end{equation}
where $a_i\in K$ for $i=0,\dots,n-1$, $a_0\neq 0$, and $y$ is an indeterminate. The theory of \cite{HS} associates with \eqref{intro:eq} a geometric object $G$, called the differential Galois group, that encodes the polynomial differential equations satisfied by the solutions of \eqref{intro:eq}. Traditionally, the following special cases have attracted special attention: $K=\mathbb{C}(z)$, and $\sigma$ is one of the following: a shift operator $\sigma:z\mapsto z+r$, where $0\neq r\in\mathbb{C}$; the $q$-dilation operator $\sigma: z\mapsto qz$, where $q\in\mathbb{C}^*$ is not a root of unity; and the Mahler operator $\sigma : z\mapsto z^p$, where $p\in\mathbb{N}_{\geq 2}$.\footnote{In the Mahler case the base field must be taken to be $K=\mathbb{C}(\{z^{1/\ell}\}_{\ell\in\mathbb{N}})$ with $\sigma(z^{1/\ell})=z^{p/\ell}$ in order for $\sigma$ to be an automorphism of $K$ and not merely an (injective) endomorphism.} More recently, the elliptic case has also attracted a lot of interest: here $K=\mathcal{M}er(E)$ is the field of meromorphic functions on the elliptic curve $E=\mathbb{C}^*/p^\mathbb{Z}$, where $p\in\mathbb{C}^*$ is such that $|p|\neq 1$ --- or equivalently, $K$ is the field of (multiplicatively) $p$-periodic meromorphic functions $f(z)$ on $\mathbb{C}^*$ such that $f(pz)=f(z)$ --- and $\sigma:f(z)\mapsto f(qz)$, where $q\in\mathbb{C}^*$ is such that $p^\mathbb{Z}\cap q^\mathbb{Z}=\{1\}$ (or equivalently, $q$ represents a non-torsion point of $E$). In each one of these four cases of interest there is a corresponding choice of derivation $\delta$ that makes $K$ into a $\sigma\delta$-field.

The main contribution of this work (Theorem~\ref{theo4}) is the development of a new set of criteria for second-order equations
\begin{equation}\label{intro:eq2} \sigma^2(y)+a\sigma(y)+by=0,\end{equation}
which guarantee that any non-zero solution $y$ of \eqref{intro:eq2} must be differentially transcendental over $K$, i.e., for any $m\in\mathbb{N}$ there is no non-zero polynomial $P\in K[y_0,y_1,\dots,y_m]$ such that $P(y,\delta(y),\dots,\delta^m(y))=0$. These criteria apply uniformly under mild conditions on the base $\sigma\delta$-field $K$ (see Definition~\ref{prop P}), which are satisfied in the four particular cases mentioned above: shift, $q$-dilation, Mahler, and elliptic. Moreover, the verification of the criteria only requires one to check whether the following auxiliary equations associated with \eqref{intro:eq2} admit any solutions in $K$: if there is no $u\in K$ such that
\begin{equation}\label{intro:riccati} u\sigma(u)+au+b=0,
\end{equation}
and there are no $g\in K$ and linear differential operator $\mathcal{L}\in\mathbb{C}[\delta]$ such that
\begin{equation}\label{intro:telescoping} \mathcal{L}\left(\frac{\delta(b)}{b}\right)=\sigma(g)-g,
\end{equation}
then every non-zero solution of \eqref{intro:eq2} must be differentially transcendental over $K$. Therefore, although we do apply the differential Galois theory for difference equations \cite{HS} in the proof that our criteria are correct, the actual verification of the criteria does not involve any prior knowledge of this theory at all. Moreover, in each of the four cases of interest mentioned above there are effective algorithms to decide whether the Riccati equation \eqref{intro:riccati} and the telescoping problem \eqref{intro:telescoping} admit solutions, for which we provide case-by-case references below. Hence these user-friendly criteria are of practical import to non-experts seeking to decide differential transcendence of solutions of second-order difference equations in many settings that arise in applications.

Indeed, we illustrate the practical applicability of our criteria in the elliptic case by proving differential transcendence of ``most" elliptic hypergeometric functions. The elliptic hypergeometric functions form a common analogue of classical hypergeometric functions and $q$-hypergeometric functions, which have been a focus of intense study in the last 200 years within the theory of special functions and are ubiquitous in physics and mathematics. The general theory of these elliptic hypergeometric functions was initiated by Spiridonov in \cite{spiridonov2016elliptic} and has been a dynamic field of research, see for instance \cite{van2007hyperbolic,fokko2009basic,
magnus2009elliptic,rains2010transformations,
rosengren2002elliptic}. In the intervening years a number of remarkable analogues of known properties and applications of classical and $q$-hypergeometric functions have been discovered for the elliptic hypergeometric functions; see \cite{spiridonov2016elliptic} for more details.

The theoretical part of our strategy to prove differential transcendence for elliptic hypergeometric functions is in the tradition of other applications of the differential Galois theory for difference equations of \cite{HS} to questions about shift difference equations \cite{arreche2017computation}, $q$-difference equations \cite{dreyfus2016functional}, deterministic finite automata and Mahler functions \cite{DHR}, lattice walks in the quarter plane \cite{dreyfus2018nature,dreyfus2017differential,dreyfus2017walks},
and shift, $q$-dilation, and Mahler difference equations in general \cite{arreche2017galois}. The Galois correspondence of \cite{HS} implies in particular that if the differential Galois group $G$ is ``large'' then there are ``few" differential-algebraic relations among the solutions of \eqref{intro:eq}. However, this theoretical strategy is only practical in the presence of algorithmic decision procedures that ensure that $G$ is indeed large enough to force any solution of \eqref{intro:eq} to be differentially transcendental. The criteria developed here in Theorem~\ref{theo4} serve to fulfill precisely this purpose. To put the novelty and usefulness of these criteria in context, let us briefly recall the state of the art in each of the four particular cases of interest mentioned above.

In the shift case, a complete algorithm to compute the differential Galois group $G$ for \eqref{intro:eq2} is developed in \cite{arreche2017computation}, based on the earlier algorithm of \cite{Hen98} to compute the non-differential Galois group $H$ of \eqref{intro:eq2} \cite{VdPS97}. Even in this case, it is still useful to have the isolated criteria of Theorem~\ref{theo4} to decide differential transcendence only, without having to compute the whole Galois group $G$ of \eqref{intro:eq2}. An algorithm for deciding whether the Riccati equation \eqref{intro:riccati} admits a solution in $K$ has been developed in \cite{Hen98}, and to decide whether there is a telescoper \eqref{intro:telescoping} one can apply \cite[Cor.~3.4]{HS}.

The situation in the $q$-dilation and Mahler cases is similar. One knows how to compute the differential Galois group $G$ for first-order equations \eqref{intro:eq} with $n=1$ by solving an associated telescoping problem (see for example \cite[Corollary~3.4]{HS} in the $q$-dilation case and \cite[Prop.~3.1]{DHR} in the Mahler case), but there is no general algorithm to compute $G$ for higher-order equations \eqref{intro:eq} with $n\geq 2$. The general criteria developed in \cite{dreyfus2016functional, DHR} for differential transcendence of solutions of \eqref{intro:eq} are valid for arbitrary $n$, but these criteria require prior knowledge of the (non-differential) Galois group $H$ of \eqref{intro:eq} \cite{VdPS97}. At present this group $H$ can only be computed in general when $n\leq 2$ by \cite{Hen97} in the $q$-dilation case and \cite{roques2018algebraic} in the Mahler case. Even when $n=2$, the criteria given here in Theorem~\ref{theo4} strictly generalize those of \cite{dreyfus2016functional, DHR}, and require no knowledge of (differential) Galois theory of difference equations for their application. Algorithms for deciding whether the Riccati equation \eqref{intro:riccati} admits solutions in $K$ have been developed in the the $q$-dilation \cite{Hen97} and Mahler \cite{roques2018algebraic} cases. Algorithms for deciding whether the telescoping problem \eqref{intro:telescoping} can be solved have also been developed in \cite[Cor.~3.4]{HS} in the $q$-dilation case and in \cite[Prop.~3.1]{DHR} in the Mahler case.

In the elliptic case, the recent algorithm developed in \cite{dreyfus2015galois} computes the (non-differential) Galois group $H$ of \eqref{intro:eq2} associated by the theory of \cite{VdPS97}, but there are no general algorithms to compute the differential Galois group $G$ for \eqref{intro:eq} for any order $n$. In spite of the relative dearth of algorithms in this case, the authors of \cite{dreyfus2018nature} were still successful in proving differential transcendence of some first-order (inhomogeneous) elliptic difference equations arising in connection with generating series for walks in the quarter plane. The criteria of Theorem~\ref{theo4} are the first to provide a test for differential transcendence that applies to second-order difference equations in the elliptic case. An algorithm for deciding whether the Riccati equation \eqref{intro:riccati} admits solutions in $K$ has been developed in \cite{dreyfus2015galois}, and criteria to decide whether the telescoping problem \eqref{intro:telescoping} can be solved has also been developed in \cite[Prop.~B.8]{dreyfus2018nature}.

The paper is organized as follows. In Section \ref{sec1}, we recall some facts about the difference Galois theory developed in \cite{VdPS97}. To a difference equation \eqref{intro:eq} is associated an algebraic group. The larger the group, the fewer the algebraic relations that exist among the solutions of the difference equation. In Section~\ref{sec:parampv}, we recall some facts about the differential Galois theory for difference equations of \cite{HS}. Here the Galois group is a linear differential algebraic group, that is, a group of matrices defined by a system of algebraic differential equations in the matrix entries.  This group encodes the polynomial differential relations among the solutions of the difference equation. In this section we prove our differential transcendence criteria for second-order difference equations \eqref{intro:eq2} in Theorem~\ref{theo4}. In Section \ref{sec3} we restrict ourselves to the situation where the coefficients of the difference equation are elliptic functions. We recall some results from \cite{dreyfus2015galois}, where the authors explain how to compute the difference Galois group of \cite{VdPS97} for order two equations with elliptic coefficients. This computation was inspired by Hendricks' algorithm, see \cite{Hen98}. In Section \ref{sec4}, we follow \cite{spiridonov2016elliptic} in defining the elliptic analogue of the hypergeometric equation \eqref{eq:hypergeo} and, under a certain genericity assumption, we prove that its nonzero solutions are differentially transcendental, see Theorem \ref{theo3}.

\section{Difference Galois theory}\label{sec1}

 For details on what follows, we refer to \cite[Chapter 1]{VdPS97}. Unless otherwise stated, all rings are commutative with identity and contain the field of rational numbers. In particular, all fields are of characteristic zero.

 A $\s$-ring (or difference ring) $(R,\s)$ is a ring $R$ together with a ring automorphism $\s : R \rightarrow R$. If $R$ is a field then $(R,\s)$ is called a $\s$-field. When there is no possibility of confusion the $\sigma$-ring $(R,\s)$ will be simply denoted by $R$. There are natural notions of $\sigma$-ideals, $\sigma$-ring extensions, $\sigma$-algebras, $\sigma$-morphisms, {\it etc}. We refer to \cite[Chapter 1]{VdPS97} for the definitions.

The ring of $\s$-constants $R^{\s}$ of the $\s$-ring $(R,\s)$ is defined by 
$$R^{\s}:=\{f \in R \ | \ \s(f)=f\}.$$
\vskip 5 pt

{\it We now let $(\K,\s)$ be a $\s$-field. We assume that the field of constants $\c:=\K^{\s}$ is algebraically closed and that the characteristic of $\K$ is~$0$.} \\

We consider a difference equation of order $n$ with coefficients in $\K$:
\begin{equation} \label{equation d ordre n}
 \sigma^{n}(y)+a_{n-1}\sigma^{n-1}(y)+\dots+a_0y=0 \text{ with } a_i \in \K \text{ and } a_0\neq 0
\end{equation}
and the associated difference system: 
\begin{equation}\label{the generic systemn}
\sigma Y=A Y, \hbox{ with } A :=\begin{pmatrix}
0&1&0&\cdots&0\\
0&0&1&\ddots&\vdots\\
\vdots&\vdots&\ddots&\ddots&0\\
0&0&\cdots&0&1\\
-a_{0}& -a_{1}&\cdots & \cdots & -a_{n-1}
\end{pmatrix} \in \mathrm{GL}_{n}(\K) 
\end{equation} 
 
By \cite[$\S$1.1]{VdPS97}, there exists a $\s$-ring extension  $(R,\s)$ of $(\K,\s)$ such that 
\begin{itemize}
\item[1)] there exists $U \in \GL_{n}(R)$ such that $\s (U) = AU$ (such a $U$ is called a fundamental matrix of solutions of (\ref{the generic systemn}));
\item[2)] $R$ is generated, as a $\K$-algebra, by the entries of $U$ and $\det(U)^{-1}$;
\item[3)] the only $\s$-ideals of $(R,\s)$ are $\{0\}$ and $R$.\\
\end{itemize}
Note that the last assumption implies $R^{\s}=\c$.
Such an $R$ is called a $\s$-Picard-Vessiot ring, or $\s$-PV ring for short, for (\ref{the generic systemn})  over $(\K,\s)$. It is unique up to isomorphism of $(\K,\s)$-algebras. Note that a $\s$-PV ring is not always an integral domain, but it is a direct sum of integral domains transitively permuted by $\s$. 

The corresponding $\s$-Galois group $\Gal(R/\K)$ of \eqref{the generic systemn} over $(\K,\s)$, or $\s$-Galois group for short, is the group of 
$(\K,\s)$-automorphisms of $R$:
$$
\Gal(R/\K):=\{ \phi \in \operatorname{Aut}(R/\K) \ | \ \s\circ\phi=\phi\circ \s \}.
$$

A straightforward computation shows that, for any~$\phi \in \Gal(R/\K)$, there exists a unique ${C(\phi) \in \GL_{n}(\c)}$ such that $\phi(U)=UC(\phi)$.
 According to \cite[Theorem~1.13]{VdPS97}, one can identify $\Gal(R/\K)$ with an {\it algebraic} subgroup $G$ of $\GL_{n}(\c)$ via the faithful representation
$$
\begin{array}{llll}
\rho :& \Gal(R/\K)&\rightarrow &\GL_{n}(\c)\\
&\hphantom{XXa}\phi &\mapsto &C(\phi).
\end{array}
$$

If we choose another fundamental matrix of solutions $U$, we find a conjugate representation.
In what follows, by ``$\sigma$-Galois group of the difference equation (\ref{equation d ordre n})'', we mean ``$\sigma$-Galois group of the difference system (\ref{the generic systemn})''. 

We shall now introduce a property relative to the base $\s$-field $(\K,\s)$, which appears in \cite[Lemma~1.19]{VdPS97}.

\begin{defi}\label{prop P}
We say that the $\s$-field $(\K,\s)$ satisfies the property $\propP$ if: 
\begin{itemize} 
\item the field $\K$ is a $\mathcal{C}^{1}$-field
 \footnote{Recall that $\K$ is a $\mathcal{C}^{1}$-field if every non-constant homogeneous polynomial $P$ over $\K$ has a non-trivial zero provided that the number of its variables is larger than its total degree.}; and
\item the only finite field extension $\L$ of $\K$ such that $\s$ extends to a field endomorphism of $\L$ is $\L=\K$.
\end{itemize}
\end{defi}

\begin{example}\label{ex1}
The following are natural examples of difference fields that satisfy property $\propP$:
\begin{trivlist}
\item \textbf{S}: Shift case with $\K=\C(z)$, $\s: f(z)\mapsto f(z+h)$, $h\in \C^{*}$. See \cite{Hen98}.
\item \textbf{Q}: $q$-difference case. $\K=\C(z^{1/*})=\displaystyle \smash{\bigcup_{\ell \in \N^{*}}} \C(z^{1/\ell})$, $\s: f(z)\mapsto f(qz)$, $q\in \C^*$, $|q|\neq 1$. See \cite{Hen97}.
\item \textbf{M}: Mahler case. $\K=\C(z^{1/*})$, $\s: f(z)\mapsto f(z^{p})$, $p\in \N_{\geq 2}$. See \cite{roques2018algebraic}. 
\item \textbf{E}: Elliptic case. See Section \ref{sec3}, and \cite{dreyfus2015galois}.
\end{trivlist}
\end{example}
The following result is due to van der Put and Singer. We recall that two difference systems $\s Y=AY$ and $\s Y=BY$ with $A,B \in \GL_{n}(\K)$ are isomorphic over $\K$ if and only if there exists $T \in \GL_{n}(\K)$ such that 
$\s(T) A=BT$. Note that for any $\sigma$-ring extension $(\mathbf{R},\sigma)$ of $(\mathbf{K},\sigma)$ and $U\in \mathrm{GL}_{n}(\mathbf{R})$, we have $\sigma(U)=AU$ if and only if $\sigma(TU)=BTU$. In what follows, for $G$ an algebraic subgroup of $\mathrm{GL}_{n}(\mathcal{C})$, we will denote by  $G(\K)\subset \mathrm{GL}_{n}(\K)$ the set of $\mathbf{K}$-rational points of $G$ (viewed as an algebraic group over $\mathcal{C}$).

\begin{theo}\label{si prop P alors}
Assume that $(\K,\s)$ satisfies property $\propP$ and let $R$ be a $\sigma$-PV ring for (\ref{the generic systemn}). 
Then the following properties relative to $G=\rho(\Gal(R/\K))$ hold:
\begin{itemize}
\item $G/G^{\circ}$ is cyclic, where $G^{\circ}$ is the identity component of $G$;
\item there exists $B \in G(\K)$ such that (\ref{the generic systemn}) is isomorphic to $\s Y=BY$ over~$\K$. 
\end{itemize}
Let $\widetilde{G}$ be an algebraic subgroup of $\GL_{n}(\c)$ such that~$A\in \widetilde{G}(\K)$. The following properties hold:
\begin{itemize}
\item $G$ is conjugate to a subgroup of $\widetilde{G}$;
\item any minimal element (with respect to inclusion) in the set of algebraic subgroups $\widetilde{H}$ of $\widetilde{G}$ for which there exists 
$T\in \GL_{n}(\K)$ such that $\s(T)AT^{-1}\in \widetilde{H}(\K)$ is conjugate to~$G$;
\item $G$ is conjugate to $\widetilde{G}$ if and only if, for any $T\in  \widetilde{G}(\K)$ and for any proper algebraic subgroup $\widetilde{H}$ of $\widetilde{G}$, 
one has that $\s(T)AT^{-1}\notin \widetilde{H}(\K)$.
\end{itemize}
\end{theo}

\begin{proof}
The proof of \cite[ Propositions~1.20 and 1.21]{VdPS97} in the special case where $\K:=\C(z)$ and $\s$ is the shift $\s(f(z)):=f(z+h)$ with $h \in \C^*$, extends {\it mutatis mutandis} to the present case.
\end{proof}

This theorem is at the heart of many algorithms to compute $\s$-Galois groups, see for example \cite{Hen97,Hen98,dreyfus2015galois,roques2018algebraic}.

\section{Parametrized Difference Galois theory}\label{sec:parampv}

\subsection{General facts}
A $(\s, \d)$-ring $(R,\s,\d)$ is a ring $R$ endowed with a ring automorphism $\s$ and a derivation $\delta : R \rightarrow R$ (this means that $\d$ is additive and satisfies the Leibniz rule $\d(ab)=a\d(b)+\d(a)b$) such that $\s \circ \delta= \delta \circ \s$. 
If $R$ is a field, then $(R,\s,\d)$ is called a $(\s, \d)$-field. When there is no possibility of confusion, we write  $R$ instead of $(R,\s, \d)$. There are natural notions of $(\s, \d)$-ideals, $(\s, \d)$-ring extensions, $(\s, \d)$-algebras, $(\s, \d)$-morphisms, {\it etc}. We refer to \cite[Section~6.2]{HS} for the definitions.

If $\K$ is a $\d$-field, and if $y_{1},\dots,y_{m}$ belong to some $\d$-$\K$-algebra, then $\K\{y_1,\dots,y_{m}\}_\d$ denotes the $\d$-algebra generated over $\K$ by $y_1,\dots,y_{m}$, and if $y_1,\dots,y_m$ belong to a $\d$-field extension of $\K$ then $\K\langle y_{1},\dots,y_{m}\rangle_{\delta}$ denotes the $\delta$-field generated over $\K$ by $y_{1},\dots,y_{m}$.\\

{\it We now let $(\K,\s,\d)$ be a $(\s,\d)$-field. We assume that the field of $\s$-constants $\c:=\K^{\s}$ is algebraically closed and that $\K$ is of characteristic~$0$.} \\

In order to apply the $(\s,\d)$-Galois theory developed in \cite{HS}, we  need to work with a base $(\s, \d)$-field $\L$ such that $\widetilde{\c}=\L^{\s}$ is $\d$-closed.\footnote{The field $\widetilde{\c}$ is called $\d$-closed if, for every (finite) set of $\delta$-polynomials $\mathcal F$ with coefficients in $\widetilde{\c}$, if the system of $\d$-equations $\mathcal F=0$ has a solution with entries in some $\delta$-field extension $\L|\widetilde{\c}$, then it has a solution with entries in $\widetilde{\c}$. Note that a $\d$-closed field is always algebraically closed.} To this end, the following lemma will be useful.

 \begin{lem}[{\cite[Lemma 2.3]{DHR}}]\label{lem:extconst}
 Suppose that $\c$ is algebraically closed and let $\widetilde{\c}$ be a $\d$-closure of $\c$ (the existence of such a $\widetilde{\c}$ is proved in \cite{Kol74}). Then the ring $\widetilde{\c} \otimes_\c \K$ is an integral domain whose fraction field $\L$ is a $(\s,\d)$-field extension of $\K$ such that ${\L^{\s}=\widetilde{\c}}$.
  \end{lem}

We still consider the difference equation \eqref{equation d ordre n} and the associated difference system \eqref{the generic systemn}. 
By  \cite[$\S$ 6.2.1]{HS}, there exists  a $(\s,\d)$-ring extension $(S,\s,\d)$ of $(\L,\s,\d)$ such that 
\begin{itemize}
\item[1)] there exists $U \in \GL_{n}(S)$ such that $\s (U)=AU$;
\item[2)] $S$ is generated, as an $\L$-$\de$-algebra, by the entries of $U$ and $\det(U)^{-1}$;
\item[3)] the only $(\s,\d)$-ideals of $S$ are $\{0\}$ and $S$.\\
\end{itemize}

Such an $S$ is called a $(\s,\d)$-Picard-Vessiot ring, or $(\s,\d)$-PV ring for short, for (\ref{the generic systemn})  over $(\L,\s,\d)$. It is unique up to isomorphism of $(\L,\s,\d)$-algebras.  Note that a $(\s,\d)$-PV ring is not always an integral domain, but it is the direct sum of integral domains that are transitively permuted by $\s$. 

The corresponding $(\s,\d)$-Galois group $\Gal^{\d}(S/\L)$ of \eqref{the generic systemn} over $(\L,\s,\d)$, or $(\s,\d)$-Galois group for short,
is the group of 
$(\L,\s,\d)$-automorphisms of~$S$:
$$
\Gal^{\d}(S/\L)=\{ \phi \in \mathrm{Aut}(S/\L) \ | \ \s\circ\phi=\phi\circ \s \text{ and } \d\circ\phi=\phi\circ \d \}.
$$
In what follows, by ``$(\s,\d)$-Galois group of the difference equation (\ref{equation d ordre n})'', we mean ``$(\s,\d)$-Galois group of the difference system (\ref{the generic systemn})''. 

A straightforward computation shows that, for any $\phi \in \Gal^{\d}(S/\L)$, there exists a unique $C(\phi) \in \GL_{n}(\widetilde{\c})$ such that $\phi(U)=UC(\phi)$.  By \cite[Proposition~6.18]{HS}, the faithful representation 
$$\begin{array}{llll}
\rho^{\delta} : & \Gal^\de(S/\L) & \rightarrow & \GL_{n}(\widetilde{\c}) \\ 
& \hphantom{XXa}\phi & \mapsto & C(\phi)
\end{array}$$
identifies $\Gal^\de(S/\L)$ with a linear differential algebraic group $G^{\delta}$, that is, a subgroup of $\GL_{n}(\widetilde{\c})$ defined by a system of $\d$-polynomial equations over $\widetilde{\c}$ in the matrix entries. If we choose another fundamental matrix of solutions $U$, we find a conjugate representation. \par

Let $S$ be a $(\s,\d)$-PV ring for \eqref{the generic systemn} over $\L$ and let $U \in \GL_{n}(S)$ be a fundamental matrix of solutions.
Then the $\L$-$\s$-algebra $R$ generated by the entries of $U$ and $\det(U)^{-1}$ is a $\s$-PV ring
for \eqref{the generic systemn} over $\L$. We can (and will) identify  $\Gal^{\d}(S/\L)$ with a subgroup of $\Gal(R/\L)$ by restricting the elements of  $\Gal^{\d}(S/\L)$ to $R$.

\begin{prop}[\cite{HS}, Proposition 2.8]\label{propo:zarclosurePPvgaloisgroup}
The group $\Gal^{\d}(S/\L)$ is a Zariski-dense subgroup of  $\Gal(R/\L)$.
\end{prop}

\begin{defi}\label{defi1}
Let $\mathbf{A}$ be a $\mathbf{K}$-$(\s,\d)$-algebra. We say that $f\in \mathbf{A}$ is differentially algebraic over $\K$, or $\K$-differentially algebraic, if there exists $m\in \N$ such that $f, \delta(f),\dots,\delta^{m}(f)$ are algebraically dependent over $\K$. Otherwise, we say that $f$ is differentially transcendental over $\K$.
\end{defi}

The following lemma will be used in the proof of Theorem~\ref{theo4}. 
 
\begin{lem}\label{lem1}
Assume that \eqref{equation d ordre n} has a nonzero $\K$-differentially algebraic solution in a $\mathbf{K}$-$(\s,\d)$ algebra $\mathbf{A}$. Then \eqref{equation d ordre n} has a nonzero $\L$-differentially algebraic solution in $S$.
\end{lem}

\begin{proof}
Since any two $(\s,\delta)$-PV rings for \eqref{equation d ordre n} over $\L$ are isomorphic, it is sufficient to prove the lemma for some $(\s,\delta)$-PV ring, not necessarily for $S$ itself. Let $f$ be a nonzero differentially algebraic solution of \eqref{equation d ordre n} in $\mathbf{A}$. Consider the $\L$-$(\s,\d)$ algebra $\mathbf{A}'=\mathbf{A} \otimes_{\K}\L$. Let us consider differential indeterminates $ X_{i,j}$, with $1\leq i\leq n, \ 2\leq j\leq n$, and let $X$ be the square matrix whose first column is $(f,\dots,\s^{n-1}(f))^{\top}$, and whose remaining columns for $2\leq j \leq n$ are $(X_{1,j},\dots,X_{n,j})^{\top}$.
We consider $T=\L \lbrace X\rbrace_\delta[\det(X)^{-1}] \subset \mathbf{A}'\lbrace X\rbrace_\delta[\det(X)^{-1}]$. This ring has a natural structure of $\L$-$(\s,\delta)$-algebra such that ${\sigma X=AX}$. If we let $\mathfrak{M}$ be a maximal $(\s,\delta)$-ideal of $T$, then the quotient $T/\mathfrak{M}$ is a $(\s,\delta)$-PV ring for $\s Y=AY$ over $\L$, the image $\bar{X}$ of $X$ in this quotient is a fundamental solution matrix for \eqref{equation d ordre n} over $\L$ and the image $\bar{f}$ of $f$ is differentially algebraic over $\L$. Let us prove that $\bar{f}$ is nonzero. Otherwise the image $\bar{X}$ in the $(\s,\delta)$-PV ring $T/\mathfrak{M}$ would have a zero first column, and therefore would not be invertible, leading to a contradiction. This concludes the proof. 
\end{proof}

\subsection{Differential transcendence criteria}

From now on, we restrict to the case $n=2$. We consider a difference equation of order two with coefficients in $\K$: 
\begin{equation} \label{equation d ordre 2}
\s^{2}(y) + a \s(y)+by=0 \text{ with } a \in \K \text{ and } b\in \K^{*}
\end{equation}
and the associated difference system: 
\begin{equation}\label{the generic system}
\s Y=AY \hbox{ with } A= \begin{pmatrix}
0&1\\
-b&-a
\end{pmatrix} \in \GL_{2}(\K).
\end{equation} 

The aim of this section is to develop a galoisian criterion for the differential transcendence of the nonzero solutions of \eqref{equation d ordre 2}. \\

{\it Recall that $\K$ is a $(\s,\d)$-field satisfying property $\propP$ such that $\c=\K^{\s}$ is algebraically closed and such that $\K$ has characteristic $0$.}\\

Let $\widetilde{\c}$ be a $\d$-closure of $\c$. According to Lemma \ref{lem:extconst}, $\widetilde{\c} \otimes_\c \K$ is an integral domain and $\L:=\mathrm{Frac}(\widetilde{\c} \otimes_\c \K)$ is a $(\s,\d)$-field extension of $\K$ such that ${\L^{\s}=\widetilde{\c}}$.
  Let $S$ be a $(\s,\d)$-PV ring for \eqref{the generic system} over $\L$ and let $R \subset S$ be a $\s$-PV ring for \eqref{the generic system} over $\L$. 
We also consider a $\s$-PV ring $\widetilde{R}$ for \eqref{the generic system} over $\K$. 

Our differential transcendence criteria are given in our main result below.

\begin{theo}\label{theo4}
Consider the second-order difference equation \eqref{equation d ordre 2}:
\[ \sigma^2(y)+a(y)+by=0,\]
where $a\in \K$ and $b\in\K^*$ and $\K$ satisfies property $\propP$. Assume the following:
\begin{enumerate}
\item there is no $u\in \K$ such that $u\sigma(u)+au+b=0$; and
\item there are no $g\in\K$ and non-zero linear differential operator $\mathcal{L}\in\c[\delta]$ such that
\[\mathcal{L}\left(\frac{\delta(b)}{b}\right)=\sigma(g)-g.\]
\end{enumerate}
Then any non-zero solution of \eqref{equation d ordre 2} in any $\mathbf{K}$-$(\s,\d)$ algebra $\mathbf{A}$ is differentially transcendental over $\K$.
\end{theo}

\begin{remark}
The most well-known historical example of a proof of differential transcendence is H\"older's proof for the differential transcendence of the Gamma function, see \cite{holder1886ueber}. An alternative proof based on $(\sigma,\delta)$-Galois theory was presented in \cite{HS}. They proved that telescoping relations like the second assumption of our theorem cannot occur for the functional equation of the Gamma function, see the proof of \cite[Corollary 3.4]{HS}.
\end{remark}

\begin{remark}
Theorem \ref{theo4} can be used to prove the differential transcendence of a large family of $q$-hypergeometric series. Indeed, they satisfy certain $q$-difference equations called $q$-hypergeometric equations, see \cite[Theorem~6]{Ro11}. The $\sigma$-Galois groups of these equations have been computed in \cite[Theorem~6]{Ro11}. In many cases, these $\sigma$-Galois groups contain $\mathrm{SL}_{2}(\C)$ so that the first assumption of our theorem is satisfied by \cite[Theorem~13]{Hen97}. It turns out that the coefficient $b$ of the $q$-hypergeometric equations is of the form $\frac{z-\alpha}{z-\beta}$ with $\alpha,\beta\in \C^*$, so \cite[Lemma~6.4]{HS} may be used to decide whether the second assumption is satisfied or not. The differential transcendence of a large family of $q$-hypergeometric series has been obtained in \cite{dreyfus2016functional}, to which we refer for more details. The latter paper improves on results that were initially proved in \cite[Example 3.14]{HS}. 
\end{remark}

Note that the first criterion of Theorem~\ref{theo4} is equivalent to the irreducibility of $\Gal(\widetilde{R}/\K)$, and may be tested algorithmically in many contexts, see \cite{Hen97,Hen98,dreyfus2015galois,roques2018algebraic}. The following lemma similarly relates the second criterion to a different largeness condition on $\Gal(\widetilde{R}/\K)$.

\begin{lem}[Proposition~2.6, \cite{DHR}]\label{lem2}
The $(\s,\d)$-Galois group of $\s y=by$ over $\L$ is a proper subgroup of $\GL_{1}(\widetilde{\c})$ if and only if there exist a nonzero linear differential operator $\mathcal{L}$ with coefficients in $\c$ and $g\in \K$ such that 
$$\mathcal{L}\left(\frac{\delta (b)}{b}\right)=\s (g)-g.$$
\end{lem}

\begin{proof}[Proof of Theorem \ref{theo4}]
Assume to the contrary that \eqref{equation d ordre 2} has a non-zero differentially algebraic solution in a $\mathbf{K}$-$(\s,\d)$ algebra $\mathbf{A}$. According to Lemma~\ref{lem1}, there exists a non-zero differentially algebraic solution $f$ of \eqref{equation d ordre 2} in $S$.\par 

By \cite[Lemma 4.1]{Hen98} combined with Theorem \ref{si prop P alors}, one of the following three cases holds
\begin{itemize}
\item $\Gal(\widetilde{R}/\K)$ is reducible.
\item $\Gal(\widetilde{R}/\K)$ is irreducible and  imprimitive.
\item $\Gal(\widetilde{R}/\K)$ contains $\SL_{2}(\mathcal{C})$.
\end{itemize}

By \cite[Lemma~13]{dreyfus2015galois}, the assumption that there is no solution in $\K$ for the Riccati equation $u\sigma(u)+au+b=0$ is equivalent to the irreducibility of $\Gal(\widetilde{R}/\K)$. Hence only the last two cases may occur. We now split our study into two cases depending on whether $\Gal(\widetilde{R}/\K)$ is imprimitive or not.

Let us first assume that $\Gal(\widetilde{R}/\K)$ is imprimitive. It follows from Theorem~\ref{si prop P alors} and \cite[Section~4.3]{Hen98} that \eqref{equation d ordre 2} is equivalent over $\K$ to
\begin{equation} \label{imprimitive eq}
\s^2(y)+ry=0
\end{equation}
for some $r\in \K^{*}$. More precisely, let $$
\s Y=BY \hbox{ with } B= \begin{pmatrix}
0&1\\
-r&0
\end{pmatrix} \in \GL_{2}(\K),
$$
be the system associated to \eqref{imprimitive eq}.
Then there exists
 $T\in \GL_{2}(\K)$ such that 
 $\sigma(T)A=BT$. Let $T=(t_{i,j})$. Since the solution space of $\sigma(Y)=AY$ and that of $\sigma(TY)=BTY$ in any given $\s$-ring extension of $\mathbf{K}$ are related by multiplication by $T$, we obtain that $t_{1,1}f+t_{1,2}\s(f)$ satisfies \eqref{imprimitive eq} with $(t_{1,1},t_{1,2})\neq (0,0)$. Let us prove that $t_{1,1}f+t_{1,2}\s(f)$ is non zero. If $t_{1,1}f+t_{1,2}\s(f)=0$, then $f\neq 0$ implies $t_{1,1}t_{1,2}\neq 0$. Then $u:=-t_{1,1}/t_{1,2}\in\K$ satisfies $\s (f)=uf$ and since $\sigma^{2}(f)+a\sigma (f)+bf=0$, we find that $u$ is a solution of the Riccati equation $u\sigma(u)+au+b=0$, which contradicts the first assumption of Theorem~\ref{theo4}.

Let us interpret \eqref{imprimitive eq} as a first-order equation with respect to $\s^2$. Since $f$ is differentially algebraic over $\K$, we have that $\s(f)$, and hence also ${t_{1,1}f+t_{1,2}\s(f)}$, are differentially algebraic over $\L$. By Lemma \ref{lem1}, applied with $n=1$, any $(\sigma^{2},\delta)$-PV ring extension for \eqref{imprimitive eq} over $\L$ contains non-zero differentially algebraic solutions of \eqref{imprimitive eq}. Since the equation has order one (with respect to $\s^2$), any such $(\s^2,\delta)$-PV ring is differentially generated by any non-zero solution and its inverse. Therefore, any $(\sigma^{2},\delta)$-PV ring extension for \eqref{imprimitive eq} over $\L$ contains only differentially algebraic elements over $\L$. By \cite[Proposition~6.26]{HS}, the $(\s^2,\delta)$-Galois group of \eqref{imprimitive eq} over $\L$ is a proper subgroup of $\GL_{1}(\widetilde{\c})$. By Lemma~\ref{lem2} there exist a nonzero $\mathcal{D}\in \c[\delta]$ and $h\in \K$ such that 
\begin{equation}\label{sig minus sig}
\mathcal{D}(\tfrac{\delta(r)}{r})=\s^2(h)-h=\s(\s(h)+h)-(\s(h)+h).
\end{equation}
 Taking the determinant in $\sigma(T)A=BT$ allows us to deduce the existence of 
  $p\in \K^{*}$ such that $b=\frac{\s(p)}{p}r$, and therefore the $(\s,\delta)$-Galois groups for ${\s(y)=ry}$ and $\s(y)=by$ are the same. Consequently, by Lemma~\ref{lem2} and the assumption on the $(\s,\d)$-Galois group of $\s y=by$ over $\L$, for any nonzero $\mathcal{D}\in\c[\delta]$ and any $g\in \K$, we have $\mathcal{D}(\frac{\delta(r)}{r})\neq\s(g)-g$. This contradicts~\eqref{sig minus sig}.
  
Assume now that $\Gal(\widetilde{R}/\K)$ is not imprimitive, so it contains $\SL_{2}(\c)$. By \cite[Proposition~2.10]{DHR}, we deduce that $$G_{m}:=\left.\left\{\begin{pmatrix}
c & 0\\
0 & c
\end{pmatrix} \ \right| \ c \in \widetilde{\c}^{*} \right\} \subset \Gal^{\d}(S/\L).$$
By way of contradiction, suppose there exists ${0 \neq P\in \L [X_{0},\dots,X_{n}]}$ with $P(f,\delta (f),\dots,\delta^{n}(f))=0$. For $c\in \widetilde{\c}^{*}$, let $\phi_{c}\in G_{m}$ with corresponding matrix $\left(\begin{smallmatrix}
c&0\\
0&c
\end{smallmatrix}\right)$.
For all $c\in \widetilde{\c}^{*}$, we find that
\begin{multline}
\phi_{c} P(f,\delta (f),\dots,\delta^{n}(f))=P(\phi_{c} (f),\phi_{c}(\delta (f)),\dots,\phi_{c}(\delta^{n}(f)))
\\=P(cf,\delta (cf),\dots,\delta^{n}(cf))=0.\label{primitive-contradiction}
\end{multline}
Let $X$ denote a new differential indeterminate, and let $\widetilde{P}\in{S}\{X\}_\delta$ denote the non-zero differential polynomial obtained by setting \[\widetilde{P}(X)=P\bigl((Xf),\delta(Xf),\dots,\delta^n(Xf)\bigr).\] It would follow from \eqref{primitive-contradiction} that $\widetilde{P}(c)=0$ for every $c\in\widetilde{\c}$, but we shall see that this is impossible.

Indeed, considering ${S}$ as a $\delta$-$\widetilde{\c}$-algebra, we may write $\widetilde{P}=\sum_{i=1}^m s_i \widetilde{P}_i$, where $s_1,\dots,s_m\in{S}$ are $\widetilde{\c}$-linearly independent and $\widetilde{P}_i\in \widetilde{\c}\{X\}_\delta$, and such that $s_i\widetilde{P}_i\neq 0$ for each $i=1,\dots,m$. Since $\delta$ does not act trivially on $\widetilde{\c}$ and $\widetilde{P}_1\neq 0$, by \cite[Corollary~II.6]{Kol73} there exists $c\in\widetilde{\c}$ such that $\widetilde{P}_1(c)\neq 0$. But since the $s_i\in{S}$ are $\widetilde{\c}$-linearly independent, this implies that $\widetilde{P}(c)\neq 0$, which contradicts~\eqref{primitive-contradiction}. \end{proof}

\section{Difference equations over elliptic curves}\label{sec3}

In this section we will be mainly interested in difference equations
\begin{equation}\label{eq:2hypergeo}
\s^{2}(y)+a\s (y)+by=0,
\end{equation}
with $a,b\in \mero_{p}$, where
\begin{itemize}
\item $\mero_{p}$ denotes the field of meromorphic functions over the elliptic curve $\C^{*}/p^{\Z}$ for some $p\in \C^{*}$ such that $\vert p \vert < 1$, {\it i.e.}  the field of meromorphic functions on $\C^{*}$ satisfying $f(z)=f(pz)$;
\item $\s$ is the automorphism of $\mero_{p}$ defined by 
$$
\s (f)(z):=f(qz)
$$
for some $q\in \C^{*}$ such that $\vert q \vert  \neq 1$ and $p^{\Z} \cap q^{\Z}=\{1\}$.
\end{itemize}
Note that this choice ensures that $\sigma$ is non cyclic.

\subsection{The base field}\label{sec31}

The difference Galois groups of linear difference equations over elliptic curves have been studied in \cite{dreyfus2015galois}. In {\it loc.\@ cit.\@} the elliptic curves are given by quotients of the form $\C/\Lambda$ for some lattice $\Lambda$. 
However, in the present work, we are mainly interested in difference equations on elliptic curves given by quotients of the form $\C^{*}/p^{\Z}$ for some $p\in \C^{*}$ such that $\vert p \vert < 1$. The translation between elliptic curves of the form $\C/\Lambda$ and elliptic curves of the form $\C^{*}/p^{\Z}$ is standard, namely by using the fact that if $\Lambda=\Z+\tau \Z$ with $\Im(\tau)>0$ and $p = e^{2 \pi \mathrm{i} \tau}$ then the map $\C\rightarrow\C^{*}:w\mapsto e^{2\pi\mathrm{i}w}$ induces an isomorphism $\C/\Lambda\simeq \C^{*}/p^\Z$.

We shall now recall some constructions and results from \cite{dreyfus2015galois}, restated in the ``$\C^{*}/p ^{\Z}$ context'' {\it via} the above identification between $\C/\Lambda$ and $\C^{*}/p ^{\Z}$. For $k\in\N^{*}$ we denote by $\C^{*}_k$ the Riemann surface of $z^{1/k}$, and we let $z_k$ be a coordinate function on each $\C^{*}_k$ such that $z_{dk}^d=z_k$ for every $d\in\N^{*}$. We will write $\C^{*}_1=\C^{*}$ and $z_1=z$.

We let $\mero_{p,k}$ denote the field of meromorphic functions on $\C^{*}_k$ satisfying $f(pz_k)=f(z_k)$, or equivalently the field of meromorphic functions on the elliptic curve $\C^{*}_k/p^\Z$. The $d$-power map $\C^{*}_{dk}\rightarrow \C^{*}_k:\xi\mapsto \xi^d$ induces an inclusion of function fields $\mero_{p,k}\hookrightarrow\mero_{p,dk}$ for each $k,d\in\N^{*}$. We denote by $\K$ the field defined by
$$
\K:= 
\bigcup_{k \geq 1} \mero_{p,k}.
$$
We endow $\K$ with the non-cyclic field automorphism 
$\s$ defined by 
\begin{equation}\label{sigma-def}
\s (f)(z_k):=f(q_kz_k)
\end{equation}
where $q_1=q\in \C^{*}$ is such that $\vert q \vert  \neq 1$ and $p^{\Z} \cap q^{\Z}=\{1\}$, and $q_k\in\C^{*}_k$ defines a compatible system of $k$-th roots of $q_1=q$ such that $q_{dk}^d=q_k$ for every $d\in\N^{*}$ (cf.~\cite[Section 2]{Hen97}).
Then~$(\K,\s)$ is a difference field and we have the following properties. 

\begin{prop}[\cite{dreyfus2015galois}, Proposition 5]
The field of constants of $(\K,\s)$ is  
$
\K^{\s}=\C
$.
\end{prop}

\begin{prop}[\cite{dreyfus2015galois}, Proposition 6]
The difference field $(\K,\s)$ satisfies property $\propP$ (see Definition~\ref{prop P}). 
\end{prop}

\begin{remark}
The field $\mero_{p}=\mero_{p,1}$ equipped with the automorphism $\s$ does not satisfy  property $\propP$. This is why we work over $(\K,\s)$ instead of $(\mero_{p},\s)$. 
\end{remark}

\begin{coro}\label{coro1}
The conclusions of Theorem \ref{si prop P alors} are valid for $(\K,\s)$.
\end{coro}
 
\subsection{Theta functions}\label{sec32}
We shall now recall some basic facts and notations about theta functions extracted from \cite[Section~3]{dreyfus2015galois} (but stated in the ``$\C^{*}/p ^{\Z}$ context'', see the beginning of the previous section).  For the proofs, we refer to \cite[Chapter I]{Mumford}.
We still consider $p\in \C^{*}$ such that $\vert p \vert < 1$. We consider the infinite product 
$$(z;p)_{\infty}=\displaystyle \prod_{j\geq 0}(1-zp^{j}).$$ 
The theta function defined by 
\begin{equation}\label{theta-def}\theta (z;p)=(z;p)_{\infty}(pz^{-1};p)_{\infty}\end{equation} 
satisfies  
\begin{equation}\label{theta-rel} \theta (pz;p)= \theta (z^{-1};p)=-z^{-1}\theta (z;p).\end{equation}
Let $\prodtheta{k}$ be the set of holomorphic functions on $\C^{*}_k$ of the form 
$$
c \displaystyle \prod_{\xi \in \C^{*}_{k}} \theta(\xi z_k)^{n_{\xi}}
$$ 
with $c \in \C^{*}$ and $(n_{\xi})_{\xi \in \C^{*}_{k}}\in \N^{(\C^{*}_{k})}$ 
with finite support. 
We denote by $\quotprodtheta{k}$ the set of meromorphic functions on $\C^{*}_{k}$ that can be written as a quotient of two elements of $\prodtheta{k}$. We have  
$$
\mero_{p,k}\subset \quotprodtheta{k}.
$$

We define the divisor $\divi{k} (f)$ of 
$
f \in \quotprodtheta{k}
$
as the following formal sum of points of $\C^{*}_{k}/p^{\Z}$: 
$$
\divi{k} (f):= \sum_{\lambda \in \C^{*}_{k}/p^{\Z}}  \operatorname{ord}_{\lambda}(f)[\lambda],
$$
where $\operatorname{ord}_{\lambda}(f)$ is the $(z_k-\xi)$-adic valuation of $f$, for an arbitrary $\xi \in \lambda$ (it follows from \eqref{theta-rel} that this valuation does not depend on the chosen $\xi \in \lambda$). 
For any $\lambda \in \C^{*}_{k}/p^{\Z}$ and any~$\xi \in \lambda$, we set 
$$
[\xi]_k:=[\lambda].
$$
Moreover, we will write 
$$
\sum_{\lambda \in \C^{*}_{k}/p^{\Z}}  n_{\lambda} [\lambda] \,\,\,\leq \sum_{\lambda \in \C^{*}_{k}/p^{\Z}}  m_{\lambda} [\lambda] 
$$
if $n_{\lambda} \leq m_{\lambda}$ for all $\lambda \in \C^{*}_{k}/p^{\Z}$.
We also introduce the weight $\weight{k} (f)$ of $f$ defined by
$$
\weight{k}(f):=\prod_{\lambda \in \C^{*}_{k}/p^{\Z}}  \lambda^{\operatorname{ord}_{\lambda}(f)} \in \C^{*}_{k}/p^{\Z}
$$
and its degree $\degr{k} (f)$ given by 
$$
\degr{k} (f):=\sum_{\lambda \in \C^{*}_{k}/p^{\Z}} \operatorname{ord}_{\lambda}(f)  \in \Z.
$$ 

\begin{example} \label{theta-div-eg} Consider $\theta=\theta(z;p)$ defined above. Then it follows from \eqref{theta-def} that $\divi{1}(\theta)=[1]$, since $\theta(z;p)$ has a zero of multiplicity one at each point of the subgroup $p^\Z\subset \C^{*}$. However, since $z=z_k^k$, we have that \[\divi{k}(\theta)=\sum_{i,j=0}^{k-1}\left[\zeta_k^i\sqrt[k]{p^j}\right],\] where $\zeta_k\in\C^{*}_k$ denotes a primitive $k$-th root of unity and $\sqrt[k]{p^j}$ is the $j$-th power of an arbitrary choice $\sqrt[k]{p}$ of $k$-th root of $p$.

Similarly, for any $f(z)\in\mero_p=\mero_{p,1}$ we have that $\divi{k}(f)=\varphi_k^{*}(\divi{1}(f))$, where $\varphi_k:\C_k^{*}/p^\Z\rightarrow \C^{*}/p^\Z$ denotes the $k$-power map and $\varphi_k^{*}$ denotes the induced pull-back map on divisors.
\end{example}

\subsection{Irreducibility of the $\s$-Galois groups}

One of the criteria of Theorem \ref{theo4} concerns the non-existence of a solution in $\K$ of a difference Riccati equation. The main tool used in this paper to address this is the following result. 

\begin{theo}[Proposition 17 in \cite{dreyfus2015galois}]\label{theo1}
Let $G$ be the $\s$-Galois group of \eqref{eq:2hypergeo} over $\K$. The following statements are equivalent: 
\begin{itemize}
\item the group $G$ is reducible;
\item the following Riccati equation has a solution in $\mero_{p,2}$:
\begin{equation}\label{ric}
u\s (u)+au+b=0.
\end{equation} 
\end{itemize}
Moreover, if $p_{1} \in \prodtheta{2} \cup \{0\}$ and $p_{2},p_{3} \in \prodtheta{2}$ are such that   
 $$
 a = \frac{p_{1}}{p_{3}} \text{ and } b = \frac{p_{2}}{p_{3}},
 $$ 
then any  solution $u \in \mero_{p,2}$ of  \eqref{ric} is of the form 
$$u = \frac{\s (r_{0})}{r_{0}}\frac{r_{1}}{r_{2}}$$
for some $r_{0},r_{1},r_{2}\in \prodtheta{2}$ such that
\begin{itemize}
\item[(i)]  $\divi{2} (r_{1}) \leq \divi{2} (p_{2})$,
\item[(ii)] $\divi{2} (r_{2}) \leq \divi{2} (\s^{-1}(p_{3}))$,
\item[(iii)] $\degr{2} (r_{1})=\degr{2} (r_{2})$, \item[(iv)]  $\weight{2} \left(r_{1}/r_{2}\right)= q_2^{\degr{2} (r_{0})} \mod p^\Z$. 
\end{itemize}
\end{theo}  

\section{Application to the elliptic hypergeometric functions}\label{sec4}

\subsection{The elliptic hypergeometric functions}\label{sec33}

We shall now introduce the elliptic hypergeometric functions following \cite{spiridonov2016elliptic}. Consider $p,q \in \C^{*}$ such that $\vert p \vert < 1$, $\vert q \vert < 1$, and $q^{\Z} \cap p^{\Z}=\{1\}$. Define 
$$(z;p,q)_{\infty}=\displaystyle \prod_{j,k\geq 0}(1-zp^{j}q^{k}) \quad\text{and}\quad
\Gamma (z;p,q)=\frac{(pq/z;p,q)_{\infty}}{(z;p,q)_{\infty}}.$$
We have  
$$\Gamma (pz;p,q)=\theta (z;q)\Gamma (z;p,q)\quad\text{and} \quad \Gamma (qz;p,q)=\theta (z;p)\Gamma (z;p,q).$$ 

For $t_{1},\ldots,t_{8}\in \C^{*}$ such that $|t_j|<1$ for each $1\leq j\leq 8$ and satisfying the balancing condition 
$\displaystyle\prod_{j=1}^{8}t_{j} =p^{2}q^{2}$, we set  
 \begin{equation}\label{v-function}V(t_{1},\dots,t_{8};p,q)=\kappa \displaystyle \int_{\mathbb{T}}\frac{\prod_{j=1}^{8}\Gamma (t_{j}z;p,q)\Gamma (t_{j}/z;p,q)}{\Gamma (z^{2};p,q)\Gamma (z^{-2};p,q)}\frac{dz}{z},\end{equation}
 where $\mathbb{T}$ denotes the positively oriented unit circle and $\kappa=\frac{(p;p)_{\infty}(q;q)_{\infty}}{4\pi \mathrm{i}}$. For $z \in \C^{*}$, we follow \cite{spiridonov2016elliptic} by setting $t_{6}=cz$, $t_{7}=c/z$, and introducing new parameters
\begin{equation}\label{epsilon-parameters}\varepsilon_{j}=\frac{q}{ct_{j}} \text{ for } j=1,\dots,5,\quad \varepsilon_{8}=\frac{c}{t_{8}},\quad \varepsilon_{7}=\frac{\varepsilon_{8}}{q},\quad c=\frac{\sqrt{\varepsilon_{6}\varepsilon_{8}}}{p^{2}}. \end{equation}
We denote $\subeps=(\varepsilon_1,\dots,\varepsilon_8)$. Note that we still have the balancing condition 
\begin{equation}\label{eq4}
\displaystyle\prod_{j=1}^{8}\varepsilon_{j} =p^{2}q^{2}.
\end{equation}
 
 \begin{defi}\label{defi2}
 The elliptic hypergeometric function $f_{\subeps}(z)$ is defined
 by the following formula 
$$f_{\subeps}(z):=\frac{V(q/c\varepsilon_{1},\dots,q/c\varepsilon_{5},cz,c/z,c\varepsilon_{8};p,q)}{\Gamma (c^{2}z/\varepsilon_{8};p,q)\Gamma (z/\varepsilon_{8};p,q)\Gamma (c^{2}/z\varepsilon_{8};p,q)\Gamma (1/z\varepsilon_{8};p,q)}. $$
\end{defi}

\begin{remark}
As explained in \cite{spiridonov2016elliptic}, the function $V(\underline{t};p,q)$ defined in \eqref{v-function} can be extended by analytic continuation, so that $\prod_{1\leq j<k\leq 8}(t_jt_k;p,q)_\infty V(\underline{t};p,q)$ is holomorphic for $t_1,\dots,t_8\in\mathbb{C}^{*}$. We should also mention for completeness that, as explained in \cite{spiridonov2016elliptic}, in Definition~\ref{defi2}  it is initially necessary to impose the constraints (expressed in terms of the old parametrization) $\sqrt{|pq|}<|t_j|<1$ for $j=1,\dots,5$ and $\sqrt{|pq|}<|q^{\pm 1}t_j|<1$ for $j=6,7,8$, which can then be relaxed by analytic continuation. These important but subtle considerations will not play a role in what follows.
\end{remark}

\subsection{The elliptic hypergeometric equation}

The elliptic hypergeometric function $f_{\subeps}(z)$ satisfies the following equation 
\begin{equation}\label{eq:hypergeo}
A(z)(y(qz)-y(z))+A(z^{-1})(y(q^{-1}z)-y(z))+\nu y(z)=0,
\end{equation}
where 
\begin{equation} \label{hypergeo-coefficients} A(z)=\frac{1}{\theta(z^{2};p)\theta(qz^{2};p)}
\displaystyle \prod_{j=1}^{8}\theta(\varepsilon_{j}z;p) \quad\text{and}\quad \nu=\displaystyle \prod_{j=1}^{6}\theta (\varepsilon_{j}\varepsilon_{8}/q;p).\end{equation}
It is easily seen that $A(pz)=A(z)$, so that the previous equation has coefficients in $\mero_{p,1}$. 

Replacing $z$ by $qz$ in \eqref{eq:hypergeo}, we obtain the following equation:  
\begin{equation}\label{eq:2hypergeo bis}
\s^{2}(y)+a\s (y)+by=0,
\end{equation}
with $a=\frac{\nu-A(qz)-A(q^{-1}z^{-1})}{A(qz)},b=\frac{A(q^{-1}z^{-1})}{A(qz)}\in \mero_{p,1}$.

\begin{remark} Note that the new parameters $\varepsilon_1,\dots,\varepsilon_8$ used in the definition of $f_{\subeps}(z)$ are not defined to be free independent parameters, since they are defined in terms of the old parameters $t_1,\dots,t_8$ (which are free parameters save for the balancing condition $\prod_{j=1}^8 t_j=p^2q^2$), and in fact one of the equations in the reparametrization \eqref{epsilon-parameters} is equivalent to $\varepsilon_8=\varepsilon_7q$.

On the other hand, the elliptic hypergeometric equation \eqref{eq:hypergeo} is defined for arbitrary parameters $\varepsilon_1,\dots,\varepsilon_8\in\mathbb{C}^{*}$, subject only to the balancing condition \eqref{eq4}, which is equivalent to imposing that the coefficients $A(z)$ and $A(z^{-1})$ actually belong to the field of elliptic functions $M_{p,1}$.

For this reason, we prove two related but distinct results on differential transcendence: (A) differential transcendence of solutions of the elliptic hypergeometric equation \eqref{eq:hypergeo}, where we think of the $\varepsilon_j$ as free parameters subject only to the balancing condition \eqref{eq4} and without imposing the additional constraint $\varepsilon_8=\varepsilon_7q$; and (B) differential transcendence of the elliptic hypergeometric functions $f_{\subeps}(z)$ where the $\varepsilon_j$ are defined in terms of the $t_j$ as in \eqref{epsilon-parameters}, and where in particular we do impose the additional constraint $\varepsilon_8=\varepsilon_7q$.
\end{remark}

Note that in case (B) above the balancing condition \eqref{eq4} for the remaining independent parameters $\varepsilon_1,\dots,\varepsilon_7$ becomes \begin{equation}\label{special-balancing} \left(\prod_{j=1}^6\varepsilon_j\right)\varepsilon_7^2=p^2q.\end{equation} In the next lemma we show that in case (B) there are no universal relations among the parameters $\varepsilon_1,\dots,\varepsilon_7$ induced from the reparametrization \eqref{epsilon-parameters}, save for formal algebraic consequences of the balancing condition \eqref{special-balancing}. This result ensures that the hypothesis in case (B) of Theorem~\ref{theo2} and Theorem~\ref{theo3} below are not vacuous.

\begin{lem}\label{lem3}
Assume that case (B) holds. Every multiplicative relation among the $\varepsilon_{1},\dots,\varepsilon_{7},p,q$ is induced by \eqref{special-balancing}, in the sense that if there are integers $\alpha_{1},\dots,\alpha_{7},m,n$ such that
$$
\displaystyle\prod_{j=1}^{7}\varepsilon_{j}^{\alpha_{j}} =p^{m}q^{n},
$$
then $\alpha_{1}=\cdots=\alpha_{6}=\alpha=n$ and $m=\alpha_7=2\alpha$ for some $\alpha\in \Z$.
\end{lem}

\begin{proof}
Let us begin to write $c$ and the $\varepsilon_{j}$ in terms of the $t_{j}$. We have $c=\sqrt{t_{6}t_{7}}$, and
$$\varepsilon_{6}=\frac{c^{2}p^{4}}{\varepsilon_{8}}=cp^{4}t_{8}=p^{4}\sqrt{t_{6}t_{7}}t_{8},
\quad \varepsilon_{7}=\frac{\varepsilon_{8}}{q}=\frac{c}{qt_{8}}=\frac{\sqrt{t_{6}t_{7}}}{qt_{8}}.$$
Assume now that there are integers $\alpha_{1},\dots,\alpha_{8},m,n$ such that
$$
\displaystyle\prod_{j=1}^{7}\varepsilon_{j}^{\alpha_{j}} =p^{m}q^{n}.
$$
Let us write this equality in term of $t_{j}$. The relation $\displaystyle\prod_{j=1}^{7}\varepsilon_{j}^{\alpha_{j}}=p^{m}q^{n}$ gives
\begin{equation}\label{eq2}
\left(\displaystyle\prod_{j=1}^{5}\frac{q^{\alpha_{j}}}{(t_{6}t_{7})^{\alpha_{j}/2}t_{j}^{\alpha_{j}}}\right) p^{4\alpha_{6}} (t_{6}t_{7})^{(\alpha_{6}+\alpha_{7})/2}t_{8}^{\alpha_{6}-\alpha_{7}}q^{-\alpha_{7}}=p^{m}q^{n}.
\end{equation} 
Using the balancing condition $\displaystyle\prod_{j=1}^{8}t_{j} =p^{2}q^{2}$, we obtain the existence of an integer $\alpha$ such that $$\alpha_{1}=\dots=\alpha_{5}=\alpha.$$
Furthermore, regarding the terms in $q$, $p$, $t_{j}$, $j=6,7$,  and $t_{8}$ respectively, we find 
\begin{multline} n-5\alpha+\alpha_{7}=-2\alpha, \quad m-4\alpha_{6}=-2\alpha,\\
-5\alpha/2+\alpha_{6}/2+\alpha_{7}/2=-\alpha,  \quad \alpha_{6}-\alpha_{7}=-\alpha.\end{multline} 
If we put the equality of the fourth relation $\alpha_{6}=\alpha_{7}-\alpha$ into the third, we obtain 
$2\alpha=\alpha_{7}$. With $\alpha_{6}=\alpha_{7}-\alpha$, we find $\alpha=\alpha_{6}$. Finally from the first and the second equality, we deduce $n=\alpha$ and $m= 2\alpha$.\end{proof}

\begin{remark}
Assume that case (B) holds. With Lemma \ref{lem3} and the relation $\varepsilon_{8}=q\varepsilon_{7}$ it follows that if there are integers $\alpha_{1},\dots,\alpha_{8},m,n$ such that
$$
\displaystyle\prod_{j=1}^{8}\varepsilon_{j}^{\alpha_{j}} =p^{m}q^{n}
$$
then $\alpha_{1}=\cdots=\alpha_{6}=\alpha=n-\alpha_8$ and $\alpha_7+\alpha_8=2\alpha=m$ for some $\alpha\in \Z$. 
\end{remark}

\subsection{Irreducibility of the $\s$-Galois group of the elliptic hypergeometric function}
From now on, we denote by $G$ the $\s$-Galois group of \eqref{eq:2hypergeo bis} over $\K$ (with respect to some $\s$-PV ring). 

\begin{theo}\label{theo2} Assume one of the two hypotheses (A) or (B) below.

(A) Every multiplicative relation among the $\varepsilon_{1},\dots,\varepsilon_{8},p,q$ is induced by \eqref{eq4}, in the sense that if there are integers $\alpha_{1},\dots,\alpha_{8},m,n$ such that
$$
\displaystyle\prod_{j=1}^{8}\varepsilon_{j}^{\alpha_{j}} =p^{m}q^{n}
$$
then $\alpha_{1}=\cdots=\alpha_{8}=:\alpha$ and $m=n=2\alpha$ for some $\alpha\in \Z$.

(B) $\varepsilon_8=\varepsilon_7q$ and every multiplicative relation among the $\varepsilon_{1},\dots,\varepsilon_{7},p,q$ is induced by \eqref{special-balancing}, in the sense that if there are integers $\alpha_{1},\dots,\alpha_{8},m,n$ such that
$$
\displaystyle\prod_{j=1}^{8}\varepsilon_{j}^{\alpha_{j}} =p^{m}q^{n}
$$
then $\alpha_{1}=\cdots=\alpha_{6}=\alpha=n-\alpha_8$ and $\alpha_7+\alpha_8=2\alpha=m$ for some $\alpha\in \Z$.

Then $G$ is irreducible. 
\end{theo}

\begin{proof}
To the contrary, assume that $G$ is reducible. 
According to Theorem~\ref{theo1}, the following Riccati equation has a solution $u \in \mero_{p,2}$~:
\begin{equation}\label{ric bis}
u\s (u)+au+b=0.
\end{equation} 
First, note that $u\in \mero_{p,2}$ is a solution of \eqref{ric bis} if and only if  ${v(\s (v)+\sigma^{-1}(a))+\sigma^{-1}(b)=0}$ with $v=\sigma^{-1}(u)\in \K$. Then to simplify the expression of the divisors of $a$ and $b$, we may replace them by $\sigma^{-1}(a)=\frac{\nu-A(z)-A(z^{-1})}{A(z)}$, $\sigma^{-1}(b)=\frac{A(z^{-1})}{A(z)}$, and consider the Riccati equation satisfied by $v$.
Consider $p_{1} \in \prodtheta{2} \cup \{0\}$ and $p_{2},p_{3} \in \prodtheta{2}$ such that 
$$
 \sigma^{-1}(a) = \frac{p_{1}}{p_{3}} \text{ and } \sigma^{-1}(b) = \frac{p_{2}}{p_{3}}.
 $$
In view of the explicit expressions for $ \sigma^{-1}(a)$ and $\sigma^{-1}(b)$, we see that we may take $p_{2}$ and $p_{3}$ such that

\begin{align*}
\divi{2}(p_2)=&\sum_{j=1}^8 \left[\sqrt{\varepsilon_j}\right]+\left[-\sqrt{\varepsilon_j}\right]+\left[\sqrt{p\varepsilon_j}\right]+\left[-\sqrt{p\varepsilon_j}\right] \\
&+\sum_{j=0}^3 \left[\sqrt[4]{p^j/q}\right]+\left[-\sqrt[4]{p^j/q}\right]+\left[\mathrm{i}\sqrt[4]{p^j/q}\right]+\left[-\mathrm{i}\sqrt[4]{p^j/q}\right]
\intertext{and}
\divi{2}(p_3)=&\sum_{j=1}^8 \left[\sqrt{1/\varepsilon_j}\right]+\left[-\sqrt{1/\varepsilon_j}\right]+\left[\sqrt{p/\varepsilon_j}\right]+\left[-\sqrt{p/\varepsilon_j}\right] \\
&+\sum_{j=0}^3 \left[\sqrt[4]{qp^j}\right]+\left[-\sqrt[4]{qp^j}\right]+\left[\mathrm{i}\sqrt[4]{qp^j}\right]+\left[-\mathrm{i}\sqrt[4]{qp^j}\right].
\end{align*}
We note for convenience that
\begin{align*}
\divi{2}(\sigma^{-1}(p_3))=&\sum_{j=1}^8 \left[\sqrt{q/\varepsilon_j}\right]+\left[-\sqrt{q/\varepsilon_j}\right]+\left[\sqrt{qp/\varepsilon_j}\right]+\left[-\sqrt{qp/\varepsilon_j}\right] \\
&+\sum_{j=0}^3 \left[\sqrt{q}\sqrt[4]{qp^j}\right]+\left[-\sqrt{q}\sqrt[4]{qp^j}\right]+\left[\mathrm{i}\sqrt{q}\sqrt[4]{qp^j}\right]+\left[-\mathrm{i}\sqrt{q}\sqrt[4]{qp^j}\right].
\end{align*}

We now consider $r_{0},r_{1},r_{2}\in\prodtheta{2}$ as in Theorem \ref{theo1}. For $i=1,2$, let \[\mathcal{S}_i:=\{\lambda\in\C_2^{*}/p^\Z \ | \ \operatorname{ord}_\lambda(r_i)\neq 0\}\] denote the support of $\divi{2}(r_i)$. For each $j\in\{1,\dots,8\}$ we let $\alpha_j\in\N$ denote the number of points in $\mathcal{S}_1$ of the form $\pm\sqrt{\varepsilon_j}$ or $\pm\sqrt{p\varepsilon_j}$. Similarly, for each $j\in\{1,\dots,8\}$ we let $\alpha_j'\in\N$ denote the number of points in $\mathcal{S}_2$ of the form $\pm\sqrt{q/\varepsilon_j}$ or $\pm\sqrt{qp/\varepsilon_j}$. We find that there exist $\ell_1,\ell_2\in\{0,1,2,3\}$ and $\gamma\in\N$ such that \[\omega_2(r_1/r_2)=\mathrm{i}^{\ell_1}\sqrt[4]{p}^{\ell_2}\prod_{j=1}^8\sqrt{\varepsilon_j}^{\alpha_j+\alpha_j'}\sqrt{q}^{-\degr{2}(r_2)}\sqrt[4]{q}^{-\gamma}=\sqrt{q}^{\degr{2}(r_0)} \mod p^\Z,\] where the second equality is obtained from property (iv) of Theorem~\ref{theo1}. After taking fourth powers we see that \begin{equation}\label{irreducibility-relation} \prod_{j=1}^8\varepsilon_j^{2\alpha_j+2\alpha_j'}=p^mq^{2\degr{2}(r_2)+\gamma+2\degr{2}(r_0)}\end{equation} for some $m\in\Z$. We now claim that $v$ is constant.

Suppose first that we are in case (A). Since every multiplicative relation among the $\varepsilon_{1},\dots,\varepsilon_{8},p,q$ is induced by \eqref{eq4}, it follows from \eqref{irreducibility-relation} that there exists $\alpha\in\N$ such that $2\alpha_j+2\alpha_j'=\alpha$ for every $j\in\{1,\dots,8\}$ and $m=2\degr{2}(r_2)+\gamma+2\degr{2}(r_0)=2\alpha$. In particular, we have that $2\degr{2}(r_2)\leq 2\alpha$. On the other hand, it follows from properties (i) and (ii) of Theorem~\ref{theo1}, respectively, that $\alpha_1+\dots+\alpha_8\leq\degr{2}(r_1)$ and $\alpha_1'+\dots+\alpha_8'\leq\degr{2}(r_2)$. We note that by property (iii) of Theorem~\ref{theo1} $2\degr{2}(r_2)=\degr{2}(r_1)+\degr{2}(r_2)$. Putting together these inequalities we obtain \[4\alpha=\sum_{j=1}^8\alpha_j+\alpha_j'\leq\degr{2}(r_1)+\degr{2}(r_2)=2\degr{2}(r_2)\leq 2\alpha.\] It follows from this that $\alpha=\degr{2}(r_1)=\degr{2}(r_2)=0$. Hence, $r_1/r_2$ is constant and \[\omega_2(r_1/r_2)=1=\sqrt{q}^{\degr{2}(r_0)} \mod p^\Z\] by property (iv) of Theorem~\ref{theo1}. Since $p^\Z\cap q^\Z=\{1\}$, we see that ${\degr{2}(r_0)=0}$ also.

Now suppose we are in case (B). Since every multiplicative relation among the $\varepsilon_{1},\dots,\varepsilon_{8},p,q$ is induced by \eqref{special-balancing}, it follows from \eqref{irreducibility-relation} that there exists $\alpha\in\N$ such that $2\alpha_j+2\alpha_j'=\alpha=2\mathrm{deg}_2(r_2)+\gamma+2\mathrm{deg}_2(r_0)-(2\alpha_8+2\alpha_8')$ for every $j\in\{1,\dots,6\}$, and $2\alpha_7+2\alpha_7'+2\alpha_8+2\alpha_8'=2\alpha=m$. It follows from the second set of equations that $2\alpha_8+2\alpha_8'\leq 2\alpha$. From this and the first set of equations it then follows that $2\mathrm{deg}_2(r_2)\leq 3\alpha$. On the other hand, it follows from properties (i) and (ii) of Theorem~\ref{theo1}, respectively, that $\alpha_1+\dots+\alpha_8\leq\degr{2}(r_1)$ and $\alpha_1'+\dots+\alpha_8'\leq\degr{2}(r_2)$. We note that by property (iii) of Theorem~\ref{theo1}, $2\degr{2}(r_2)=\degr{2}(r_1)+\degr{2}(r_2)$. Putting together these inequalities we obtain \[4\alpha=\sum_{j=1}^8\alpha_j+\alpha_j'\leq\degr{2}(r_1)+\degr{2}(r_2)=2\degr{2}(r_2)\leq 3\alpha.\] It follows from this that $\alpha=\degr{2}(r_1)=\degr{2}(r_2)=0$. Hence, $r_1/r_2$ is constant and \[\omega_2(r_1/r_2)=1=\sqrt{q}^{\degr{2}(r_0)} \mod p^\Z\] by property (iv) of Theorem~\ref{theo1}. Since $p^\Z\cap q^\Z=\{1\}$, we see that ${\degr{2}(r_0)=0}$ also. 

It follows from the above in either of the cases (A) or (B) that $v\in \C^{*}$ is constant. Therefore \eqref{ric bis} can be rewritten as
\begin{equation}\label{eq1}
v^{2}A(z)+v(\nu-A(z)-A(z^{-1}))+A(z^{-1})=0, 
\end{equation}
{\it i.e.} 
\begin{equation}\label{eq1bis}
(v^{2}-v)A(z)+v\nu=(v-1)A(z^{-1}).  
\end{equation}
But since $\sqrt{q}^{-1}$ is a pole of $A(z)$ but not of $A(z^{-1})$ and, on the other hand, $\sqrt{q}$ is a pole of $A(z^{-1})$ but not of $A(z)$, we obtain that $v^{2}-v=v-1=v\nu=0$. So we must have $\nu=0$. On the other hand, we see from the definition of $\nu$ in \eqref{hypergeo-coefficients} that $\nu=0$ if and only if $\varepsilon_j\varepsilon_8=qp^\ell$ for some $\ell\in\mathbb{Z}$ and $j=1,\dots,6$, which is ruled out by our hypotheses in both cases (A) and (B). 
This contradiction concludes the proof that $G$ is irreducible.
\end{proof}

\subsection{Differential transcendence of the elliptic hypergeometric functions}
We may equip $(\K,\s)$ with the classical derivation $\d:=z\frac{d}{dz}$ as in \cite[Section~3.1]{dreyfus2018nature}. Note that $\d$ commutes with $\s$. Let $\widetilde{\C}$ be the $\d$-closure of $\C$. Following Lemma \ref{lem:extconst}, we may consider $\L:=\mathrm{Frac} (\K\otimes_{\C}\widetilde{\C})$ and we have $\L^{\s}= \widetilde{\C}$. Recall that $f_{\subeps}(z)$ is meromorphic on $\C^{*}$, and note that the field of meromorphic functions on $\C^*$ is a $(\sigma,\delta)$-extension of $\K$.

\begin{theo}\label{theo3} Assume one of the two hypotheses (A) or (B) below.

(A) Every multiplicative relation among the $\varepsilon_{1},\dots,\varepsilon_{8},p,q$ is induced by \eqref{eq4}, in the sense that if there are integers $\alpha_{1},\dots,\alpha_{8},m,n$ such that
$$
\displaystyle\prod_{j=1}^{8}\varepsilon_{j}^{\alpha_{j}} =p^{m}q^{n}
$$
then $\alpha_{1}=\cdots=\alpha_{8}=:\alpha$ and $m=n=2\alpha$ for some $\alpha\in \Z$.

(B) $\varepsilon_8=\varepsilon_7q$ and every multiplicative relation among the $\varepsilon_{1},\dots,\varepsilon_{7},p,q$ is induced by \eqref{special-balancing}, in the sense that if there are integers $\alpha_{1},\dots,\alpha_{8},m,n$ such that
$$
\displaystyle\prod_{j=1}^{8}\varepsilon_{j}^{\alpha_{j}} =p^{m}q^{n}
$$
then $\alpha_{1}=\cdots=\alpha_{6}=\alpha=n-\alpha_8$ and $\alpha_7+\alpha_8=2\alpha=m$ for some $\alpha\in \Z$.

Then any non-zero solution to \eqref{eq:hypergeo} is differentially transcendental over~$\K$.
\end{theo}

\begin{proof}
We apply the criteria of Theorem \ref{theo4}. We proved in Theorem~\ref{theo2} that $G$ is irreducible, which by \cite[Lemma 13]{dreyfus2015galois} is equivalent to the non-existence of a solution $u\in\K$ to the Riccati equation
\[
u\sigma(u)+au+b=0.
\]
It remains to show that there is no nonzero linear differential operator $\mathcal{L}$ in $\delta$ with coefficients in $\C$ and $g\in \K$ such that $$\mathcal{L}\left(\frac{\delta b}{b}\right)=\s (g)-g. $$

Let $k\in \N^{*}$ such that $g\in \mero_{p,k}$ and consider $b$ as an element of $\mero_{p,k}$. Let $\omega\in \C^{*}_{k}/p^{\Z}$ be a zero or a pole of $b$. Then it is a pole of $\frac{\delta b}{b}$. Since $\mathcal{L}$ has constant coefficients, we get that $\omega$ is also a pole of $\mathcal{L}\left(\frac{\delta b}{b}\right)$.
Therefore, $\omega$ is a pole of $\s(g)-g$ and hence also a pole of $\s(g)$ or of $g$. Furthermore, $\s(g)-g$ has at least two distinct poles $\omega',\omega''\in\C^{*}_k/p^{\Z}$ such that $\omega \equiv \omega' \equiv \omega''\mod q_k^{\Z}$, where $q_k\in\C^{*}_k$ is as in \eqref{sigma-def}. These $\omega'$ and $\omega''$ are poles of $\frac{\delta b}{b}$, and hence zeros or poles of $b$ has well. We have proved that, for every $\omega\in \C^{*}_{k}/p^{\Z}$ that is a pole or zero of $b$, there exists $\ell\in \Z_{\neq 0}$ such that $\omega q_{k}^{\ell}$ is a pole or zero of $b$. \par 
Let us now consider $b$ as an element of $\mero_{p,1}$. From the preceding, we deduce that for every $\omega\in \C^{*}/p^{\Z}$, pole or zero of $b$, there exists $\ell\in \Z_{\neq 0}$ such that $\omega q^{\ell}$ is a pole or zero of $b$.
We will use this to find a contradiction. Note that the set of zeros or poles of $b=\frac{\theta(q^{2}z^{2};p)\theta(q^{3}z^{2};p)}{\theta(q^{-2}z^{-2};p)\theta(q^{-1}z^{-2};p)}\times \displaystyle \prod_{j=1}^{8}\frac{\theta(\varepsilon_{j}q^{-1}z^{-1};p)}{\theta(\varepsilon_{j}qz;p)}$, seen as an element of $\mero_{p,1}$, is included in $$\mathcal{S}=\{q^{-1}\varepsilon_{1}^{\pm 1},\ldots,q^{-1}\varepsilon_{8}^{\pm 1}, \pm q^{-1/2}, \pm q^{-1/2}\sqrt{p}, \pm q^{-3/2}, \pm q^{-3/2}\sqrt{p}\} \mod p^{\Z}.$$ 
Let us prove that the elements of $\mathcal{S}$ are all distinct. To see this, note that if any two elements of $\mathcal{S}$ were the same modulo $p^\mathbb{Z}$ then we would find a non-trivial multiplicative relation satisfied by at most four elements among $p,q,\varepsilon_{1},\dots,\varepsilon_{8}$. This contradicts the hypothesis in both cases (A) and (B). Therefore, no simplifications occur and $\mathcal{S}$ is exactly the set of zeros or poles of $b$.
It suffices to show that for all $\ell\in \Z_{\neq 0}$, we have ${\mathcal{S} \cap \{ q^{\ell} q^{-1}\varepsilon_{1} \mod p^\Z\}=\varnothing}$. Let $\ell\in \Z$ such that $\mathcal{S} \cap \{q^{\ell} q^{-1}\varepsilon_{1} \mod p^\Z\}\neq \varnothing$.  If $\ell\neq 0$, then we again find a non-trivial multiplicative relation satisfied by at most four elements among $p,q,\varepsilon_{1},\dots,\varepsilon_{8}$. In either case (A) or case (B) this contradiction to the hypothesis concludes the proof. \end{proof}

\section*{Acknowledgements}

The authors are very thankful to the referee for their helpful comments and suggestions, and for providing us with a better and simplified argument for the last part of the proof of Theorem~\ref{theo4}. The first author also thanks Professor Spiridonov for helpful discussions on an earlier version of this work, which led to the important splitting of Theorems~\ref{theo2} and \ref{theo3} into cases (A)~and~(B).

\newcommand{\etalchar}[1]{$^{#1}$}
\def\cprime{$'$} \def\polhk#1{\setbox0=\hbox{#1}{\ooalign{\hidewidth
  \lower1.5ex\hbox{`}\hidewidth\crcr\unhbox0}}} \def\cprime{$'$}

\end{document}